\documentclass[9pt, a4paper]{amsart}

\usepackage{amsfonts}
\usepackage{amssymb}
\usepackage{amsmath, amsthm,color}
\usepackage{hyperref}
\usepackage{pdfsync}
\usepackage{comment}
\usepackage{calrsfs}
\DeclareMathAlphabet{\pazocal}{OMS}{zplm}{m}{n}

\DeclareMathAlphabet{\mathcal}{OMS}{cmsy}{m}{n}
\SetMathAlphabet{\mathcal}{bold}{OMS}{cmsy}{b}{n}
\newcommand{\R}{\mathbb{R}}
\newcommand{\C}{\mathbb{C}}

\newtheorem*{thmnonum1}{Theorem A}
\newtheorem*{thmnonum2}{Theorem B}
\newtheorem{theorem}{Theorem}
\newtheorem{corollary}{Corollary}
\newtheorem{lemma}{Lemma}
\newtheorem{proposition}{Proposition}
\newtheorem{remark}{Remark}

\numberwithin{equation}{section}

\begin{document}
\title{$(P,Q)$ complex hypercontractivity}
\author[P.~Ivanisvili]{Paata Ivanisvili}
\author[P.~Kalantzopoulos]{Pavlos Kalantzopoulos}

\address{Department of Mathematics, University of California Irvine, Irvine, CA
}
\email{pivanisv@uci.edu \textrm{(P.\ Ivanisvili)}}

\address{Department of Mathematics, University of California Irvine, Irvine, CA
}
\email{pkalantz@uci.edu  \textrm{(P.\ Kalantzopoulos)}}
\makeatletter
\@namedef{subjclassname@2010}{
  \textup{2010} Mathematics Subject Classification}
\makeatother
\begin{abstract} 
Let $\xi$ be the standard normal random vector in $\mathbb{R}^{k}$. Under some mild growth and smoothness assumptions on any increasing   $P, Q : [0, \infty)  \mapsto [0, \infty)$  we show $(P,Q)$ complex hypercontractivity  
$$
Q^{-1}(\mathbb{E} Q(|T_{z}f(\xi)|))\leq P^{-1}(\mathbb{E}P(|f(\xi)|))
$$
holds for all polynomials $f:\mathbb{R}^{k} \mapsto \mathbb{C}$, where $T_{z}$ is the hermite semigroup at complex parameter $z, |z|\leq 1$, if and only if 
\begin{align*}
\left|\frac{tP''(t)}{P'(t)}-z^{2}\frac{tQ''(t)}{Q'(t)}+z^{2}-1\right|\leq \frac{tP''(t)}{P'(t)}-|z|^{2}\frac{tQ''(t)}{Q'(t)}+1-|z|^{2}
\end{align*}
holds for all $t>0$ provided that $F''>0$, and $F'/F''$ is concave, where $F = Q\circ P^{-1}$.  This extends Hariya's result from real to complex parameter $z$. Several old and new applications are presented for different choices of $P$ and $Q$. The proof uses heat semigroup arguments, where we find  a certain map $C(s)$, which interpolates the inequality at the endpoints. The map $C(s)$ itself is composed of  four heat flows running together at different times. 
  \end{abstract}
\maketitle

\section{Introduction}\label{Introduction}
We denote the  Gaussian  on $\R^k$  by
\begin{equation*}
	d\gamma(x)=\frac{1}{(2\pi)^{\frac{k}{2}}}e^{-|x|^2/2}\,dx,
\end{equation*}
where $|x|$ is the Euclidean norm of $x=(x_1,\ldots,x_k)\in\R^k$. For $p>0$ we set 
\begin{equation*}
	\|f\|_p:=\left(\int_{\R^k}|f|^p\,d\gamma\right)^{\frac{1}{p}} 
\end{equation*}
for all Borel measurable functions $f :\mathbb{R}^{k} \mapsto \mathbb{C}$.  Given a multiindex $\alpha=(\alpha_1,\ldots,\alpha_k)\in(\mathbb{Z}_+)^k$ we define probabilist's Hermite polynomials on $\mathbb{R}^{k}$ as follows 
\begin{equation*}
H_{\alpha}(x)=\int_{\mathbb{R}^{k}}(x+iy)^{\alpha}d\gamma(y),
\end{equation*}
where $w^{\alpha} := w_{1}^{\alpha_{1}}\ldots w_{k}^{\alpha_{k}}$ for any $w=(w_{1}, \ldots, w_{k}) \in \mathbb{C}^{k}$. 
The family $\{H_{\alpha}\}_{\alpha \in (\mathbb{Z}_{+})^{k}}$ forms an orthogonal basis in $L^{2}(\mathbb{R}^{k}, d\gamma)$ equipped with the natural innear product structure. 

Any polynomial $f:\R^k\to \C$ admits a representation 
\begin{equation}\label{P expans Ha}
	f(x)=\sum_{|\alpha|\leq \deg f}c_{\alpha}H_{\alpha}(x),
\end{equation}
for some $c_{\alpha}\in \C$,  where $|\alpha|=\alpha_1+\ldots+\alpha_k$.  For any  complex number $z\in \C$, $|z|\leq 1$,    Mahler transform $T_{z}$ is an operator that acts on $f$ as 
\begin{equation*}
	T_zf(x)=\sum_{|\alpha|\leq\deg f}z^{|\alpha|}c_{\alpha}H_{\alpha}(x).
\end{equation*}
Clearly $T_{0}f(x) = \int_{\mathbb{R}^{k}} f d\gamma$, and $T_{1}f(x)=f(x)$. When $z=r$, $r \in [-1,1]$, the operator $$
T_{r}f(x) = \int_{\mathbb{R}^{k}} f(xr+\sqrt{1-r^{2}}y) d\gamma(y)
$$
 is called the noise operator (or Ornstein--Uhlenbeck semigroup when parametried by $r=e^{-t}$). Perhaps one of the most fundamental results in Gauss space is the hypercontractivity due to Bonami \cite{Bonami}, Nelson \cite{Nelson}, and Gross \cite{Gross} which says that  for any $p\geq 1$ we have 
\begin{align}\label{har1}
\| T_{r} f \|_{(p-1)/r^{2}+1} \leq \|f\|_{p} \quad \text{holds for all} \quad r \in [-1,1].
\end{align}

Recently \cite{Hariya} Hariya  showed that for any $r \in (0,1]$
\begin{align}\label{har2}
    r^{2}\log\left(\int_{\mathbb{R}^{k}}\exp\left(\frac{T_{r} h}{r^{2}}\right) d\gamma\right) \leq \log\left( \int_{\mathbb{R}^{k}}\exp(h) d\gamma \right)
\end{align}
holds for any Borel measurable $h :\mathbb{R}^{k} \mapsto \mathbb{R}$ with $\|h\|_{1}+\|e^{h}\|_{1}<\infty$. Moreover, using stochastic calculus Hariya unified both of the statements (\ref{har1}) and (\ref{har2}) into one compact and elegant inequality.  
\begin{thmnonum1}[Hariya \cite{Hariya}]\label{harth}
Let nonnegative $P \in C([0, \infty))\cap C^{2}((0, \infty))$ be sucht that $P'>0$ on $(0, \infty)$. Assume 
\begin{align*}
P''(t)>0,  \quad \text{and} \quad t \mapsto \frac{P'(t)}{P''(t)} \quad \text{is concave on} \quad (0, \infty). 
\end{align*}
Then for any $r \in (0,1]$, and $Q(t) = \int_{0}^{t} (P'(s))^{1/r^{2}}ds$, we have
\begin{align}\label{har3}
Q^{-1}\left(\int_{\mathbb{R}^{k}} Q(T_{r}f) d\gamma\right) \leq P^{-1}\left( \int_{\mathbb{R}^{k}} P(f)d\gamma\right)
\end{align}
holds for all measurable $f\geq 0$ with  $\| P(f)\|_{1}<\infty$. 
\end{thmnonum1}

Notice that the choices $P(t)=t^{p-1},$ and $P(t)=e^{t},$ recover the statements (\ref{har1}) and (\ref{har2}) correspondingly. In \cite{Hariya} Hariya illustrated applications of his Theorem~\ref{har3} for various choices of $P$. By {\em differentiating} the inequality (\ref{har3}) at point $r=1$, Hariya obtained a {\em general} $P$-log-Sobolev ineqiality, we refer the reader to \cite{Hariya} for full details.

The goal of this paper is to extend Hariya's result to complex parameter $r$. 
\begin{theorem}\label{mainth1}
     Let nonnegative $Q, P \in C^{2}((0, \infty))\cap C([0, \infty))$ be such that $Q', P'>0$ on $(0, \infty)$, and both $P$ and $F:=Q\circ P^{-1}$ satisfy growth  condition (\ref{growthC}). If  
 \begin{align}\label{convexityp}
    F''(t)>0 \quad \text{and} \quad t \mapsto F'(t)/F''(t) \quad \text{is concave on} \quad (0,\infty), 
\end{align}
and the local condition 
\begin{align}\label{localpp}
    	\left|\frac{tP''(t)}{P'(t)}-z^{2}\frac{tQ''(t)}{Q'(t)}+z^{2}-1\right|\leq \frac{tP''(t)}{P'(t)}-|z|^{2}\frac{tQ''(t)}{Q'(t)}+1-|z|^{2}
\end{align}
 holds for all $t>0$, and some $z \in \mathbb{C}$, $|z|\leq 1$,   then we have
 \begin{align}\label{globalpp}
		Q^{-1}\left(\int_{\mathbb{R}^{k}} Q(|T_{z} f|) d\gamma \right) \leq P^{-1}\left(\int_{\mathbb{R}^{k}} P(|f|)d\gamma \right).
 \end{align}
 holds for all polynomials $f : \mathbb{R}^{k} \mapsto \mathbb{C}$, and all $k\geq 1$.
\end{theorem}
\begin{remark}
 Theorem~\ref{mainth1} will be proved, see Remark~\ref{remarka1},  with (\ref{convexityp}) replaced by a weaker assumption 
 \begin{align}\label{wconvexity}
     (t,y) \mapsto \frac{F''(t)}{F'(t)}y^{2} \quad \text{is concave on} \quad (0,\infty)\times \mathbb{R},
 \end{align}
 which also covers the case when $F''\equiv 0$. However, it will be explained, see Proposition~\ref{utver1},  that if $F''>0$ then (\ref{convexityp}) and (\ref{wconvexity}) are equivalent. 
\end{remark}
We will see that when $z$ is real Theorem~\ref{mainth1} recovers Hariya's result under a mild growth assumption on $P$ together with $P \in C^{4}((0,\infty))$. We will also see that for $P(t)=t^{p}$ and $Q(t)=t^{q}$ and $z$ is complex,   Theorem~\ref{mainth1} recovers complex hypercontractivity in Gauss space due to Weissler \cite{Wei} (for all exponents $1<p\leq q<\infty$ except $2<p\leq q \leq 3$ and its dual),  Janson \cite{Jans}, Lieb \cite{Lieb},  and Epperson \cite{Epp} (Epperson  had a gap in his proof which was later fixed by Janson), and in particular Hausdorff--Young inequality with sharp constants due to Beckner \cite{Beckner}.

Our approach is inspired by the work Ivanisvili-Volberg \cite{IvanVol},  it is based on the observation that under the assumptions \eqref{convexityp} and (\ref{localpp}), the map 
\begin{align}
&C(s)=\int_{\R^k}F\left(\int_{\R^k}P\Big(\Big|\int_{\R^k}\int_{\R^k}\right. \label{otob}\\
&\left.\sum_{|\alpha|\leq \mathrm{deg}(f)}c_{\alpha}(\sqrt{s}(u+iv)+z\sqrt{1-s}(x+iy))^{\alpha}\,d\gamma(v)\,d\gamma(y)\Big|\Big)\,d\gamma(u)\right)\,d\gamma(x) \nonumber
\end{align}
is nondecreasing on $[0,1]$, and it interpolates the inequality (\ref{globalpp}) at the endpoints $s=0$ and $s=1$. The verification of $C'(s)\geq 0$  is done by a direct differentiation\footnote{In general $C'(s)$ may not  exist, so in the formal proof we work with a {\em smoothened} version of $C(s)$.}, see Section~\ref{Proof} for the details.  In Section~\ref{construction} via heuristic arguments we  explain how the map $C(s)$ was discovered. 

The assumption (\ref{convexityp}) may seem a bit artificial. It simply means  that  
\begin{align*}
	\mathfrak{M}_{F}(h) = F^{-1}\left(\int_{\mathbb{R}^{n}} F (h) d\gamma\right) \quad \text{is convex}
\end{align*} 
 as a functional acting on nonnegative functions $h$. For example, Theorem 106 in \cite{HLP}, page 88,  says that if $F, F', F''>0$,  then the functional $\mathfrak{M}_{F}$ is convex if and only if $F'/F''$ is concave. This assumption also appeared in the statement of Hariya's theorem, and one should think of it as {\em substitute} of Minkowski's inequality used in the proofs of complex hypercontractivity on the hypercube Beckner \cite{Beckner}, Weissler \cite{Wei}, see also Section~\ref{construction} for more details.

The local condition (\ref{localpp}) turns out to be necessary for the inequality (\ref{globalpp}) to hold for all polynomials $f$, this is explained in Section~\ref{localpp}. Another way to think about our Theorem~\ref{mainth1} is that under the assumption (\ref{convexityp}) the two inequalities, the local  (\ref{localpp}) and the global (\ref{globalpp}),  are equivalent to each other.

\section{Applications}
\subsection{Hariya's inquality}
Let $P$ satisfy the assumptions in Hariya's Theorem~\ref{harth}. Additionally assume $P \in C^{4}((0,\infty))$ and $P$ satisfies the growth condition (\ref{growthC}). 

Let $z=r \in (0,1)$. In this case the local condition (\ref{localpp}) simplifies to 
\begin{align}\label{haha}
\frac{Q''(t)}{Q'(t)}r^{2}\leq \frac{P''(t)}{P'(t)}
\end{align}
holds for all $t>0$. The choice $Q(t) = \int_{0}^{t} (P'(s))^{1/r^{2}}ds$ gives equality in (\ref{haha}). 

Clearly $F(t)=Q(P^{-1}(t))$ satisfies growth condition (\ref{growthC}). Next we check (\ref{convexityp}). We have 
\begin{align*}
&F'(t) = \frac{Q'(P^{-1}(t))}{P'(P^{-1}(t))}=P'(P^{-1}(t))^{1/r^{2}-1},\\
&F''(t) = \left(\frac{1}{r^{2}}-1\right)P'(P^{-1}(t))^{1/r^{2}-3} P''(P^{-1}(t)).
\end{align*}
Thus $F''(t)>0$ follows from $P''(t)>0$. Notice that 
\begin{align*}
\frac{F'(t)}{F''(t)}=\frac{r^{2}}{1-r^{2}}\frac{[P'(P^{-1}(t))]^{2}}{P''(P^{-1}(t))}.
\end{align*}
Let $T = P^{-1}(t)$. We have 
\begin{align*}
\frac{\partial}{\partial t}\frac{[P'(P^{-1}(t))]^{2}}{P''(P^{-1}(t))} = \frac{2P'(T)P''(T)}{P''(T)P'(T)} - \frac{P'(T)^2 P'''(T)}{P''(T)^{2} P'(T)} = 2-\frac{P'(T) P'''(T)}{P''(T)^{2}}.
\end{align*}
On the other hand we have 
\begin{align*}
\frac{\partial}{\partial t} \frac{P'(t)}{P''(t)} = 1 - \frac{P'(t)P'''(t)}{P''(t)^{2}}.
\end{align*}
Since $\frac{\partial }{\partial t} T  = 1/P'(T) >0$ it follows that the sign of $\frac{\partial^{2}}{\partial t^{2}} \frac{P'(t)}{P''(t)}$ coincides with the sign of $\frac{\partial^{2}}{\partial t^{2}}\frac{F'(t)}{F''(t)}$, hence, by Theorem~\ref{mainth1} the inequality (\ref{globalpp}) holds for all polynomials $f : \mathbb{R}^{k} \mapsto \mathbb{C}$. 

\subsection{Complex hypercontractivity,  
Babenko--Beckner, and Hirschman}
Let $1\leq p \leq q$. Choose $Q(t)=t^{q}$ and $P(t)=t^{p}$. Then both $F(t) = t^{q/p}$ and $P(t)=t^{p}$ satisfy the assumptions of Theorem~\ref{mainth1}. The local condition (\ref{localpp}) takes the form 
\begin{align}\label{weis1}
|p-2-z^{2}(q-2)|\leq p-|z|^{2}q.
\end{align}
Thus we obtain that $\|T_{z} f\|_{q} \leq \|f\|_{p}$ holds for all polynomials $f$ if and only if (\ref{weis1}) holds. This recovers the complex hypercontractivity due to Weissler \cite{Wei}, Epperson \cite{Epp} (which had a gap in the proof and later was fixed by Janson \cite{Jans} via different arguments), and Lieb~\cite{Lieb}. 

If we choose $q=\frac{p}{p-1}$, $z=i\sqrt{p-1}$, where $p\in (1,2]$, then we have equality in (\ref{weis1}), and hence (\ref{weis1}) holds true. This particular case was proved by Beckner~\cite{Beckner} who also showed that the inequality 
\begin{align}\label{HY}
\| T_{i\sqrt{p-1}}f\|_{q}\leq \|f\|_{p}
\end{align}
is (after a change of variables) the same as Hausdorff--Young inequality 
\begin{align}\label{HY1}
 \left(\int_{\mathbb{R}^{k}}|\widehat{g}(y)|^{q}dy\right)^{1/q} \leq   \left(\frac{p^{1/p}}{q^{1/p}}\right)^{k/2}\left(\int_{\mathbb{R}^{k}}|g(y)|^{p}dy\right)^{1/p}, 
\end{align}
with the sharp constant, where $\widehat{g}(y) = \int_{\mathbb{R}^{k}}e^{-2\pi i x\cdot y}g(x)dx$, and $x\cdot y$ denotes innear producr in $\mathbb{R}^{k}$. For $p=2$ we have equality in (\ref{HY1}). Differentiating the estimate (\ref{HY1}) at point $p=2$ gives Hirschman's uncertainty principle $H(|f|^{2})+H(|g|^{2})\geq \log(e/2)$, where $H(|f|^{2})$ denotes the Shannon entropy $-\int_{\mathbb{R}} |f|^{2}(x) \log(|f|^{2}(x))dx$ , see \cite{Beckner}. 

\subsection{The degenerate case $F'' \equiv 0$} Let us choose $Q=P$ in Theorem~\ref{mainth1}. Then $F(t)=t$, and hence $F''\equiv 0$. The local condition (\ref{localpp}) takes the form 
\begin{align}\label{degloc}
|1-z^{2}| \left|\frac{tP''(t)}{P'(t)}-1\right|\leq \left(\frac{tP''(t)}{P'(t)}+1 \right)(1-|z|^{2})
\end{align}
holds for all $t>0$. Consider the case $z\neq \pm 1$. Then (\ref{degloc}) implies that $P''(t)>0$ on $(0,\infty)$. 
Let $K =\frac{tP''(t)}{P'(t)}$ and $L = \frac{|K-1|}{K+1}$. Then (\ref{degloc}) rewrites as 
\begin{align}\label{deglocp}
1-|z|^{2} \geq L|1-z^{2}|.
\end{align}
Squaring and subtracting $(1-|z|^{2})^{2}L^{2}$ from boths sides of the inequality (\ref{deglocp}), we obtain 
\begin{align*}
(1-|z|^{2})^{2}(1-L^{2})\geq L^{2} (|1-z^{2}|^{2}-(1-|z|^{2})^{2})=L^{2}\left(\frac{z-\bar{z}}{i}\right)^{2}.
\end{align*}
So taking the square root in the last inequality we obtain 
\begin{align}\label{deglocpp1}
1-|z|^{2} \geq \frac{2|L|}{\sqrt{1-L^{2}}}|\Im z|.
\end{align}
Next, notice that 
\begin{align*}
\frac{2|L|}{\sqrt{1-L^{2}}} = \sqrt{K+\frac{1}{K}-2},
\end{align*}
So if we let 
\begin{align*}
c_{P} = \sup_{t \in (0,\infty)} \left\{\frac{tP''(t)}{P'(t)}+\frac{P'(t)}{tP''(t)}\right\}
\end{align*}
then (\ref{deglocpp1}) for all $t>0$ is the same as 
\begin{align}\label{lens}
|2z \pm i \sqrt{c_{P}-2}|\leq \sqrt{c_{P}+2},
\end{align}
Thus we recover the result obtained in \cite{IvaniEsk}
\begin{thmnonum2}[Eskenazis--Ivanisvili, \cite{IvaniEsk}]
Let $z \in \mathbb{C}$ and a nonnegative $P \in C([0,\infty))\cap C^{3}((0,\infty))$ be such that $P', P''>0$ on $(0, \infty)$, and $P$ satisfies the growth condition (\ref{growthC}). Then 
\begin{align}\label{esken}
\int_{\mathbb{R}^{k}} P(|T_{z}f|) d\gamma \leq \int_{\mathbb{R}^{k}} P(|f|) d\gamma
\end{align}
holds  for all polynomials $f :\mathbb{R}^{k} \mapsto \mathbb{C}$ if and only if  $z$ belongs to the ``lens'' domain (\ref{lens}). 
\end{thmnonum2}
The lens domain (\ref{lens}) represents intersection of two disks symmetric with respect to the real axes, centered on the imaginary axes passing at the points $z=\pm 1$, therefore, (\ref{lens}) implies $|z|\leq 1$, and  we just explained how ``if'' part follows from Theorem~\ref{mainth1}.  The ``only if'' part  follows from Section~\ref{localpp} (see also \cite{IvaniEsk}). 

\textbf{Example of $P(t)=\int_{0}^{t}s^{p}\log(1+s)ds$ with $p>0$}. One can check that
\begin{equation*}
	\frac{tP''(t)}{P'(t)}=p+\frac{t}{(t+1)\log(t+1)}.
\end{equation*}
Differentiating we get
\begin{equation*}
		\left(\frac{tP''(t)}{P'(t)}\right)'=\frac{(t+1)\log(t+1)-t\log(t+1)-t}{((t+1)\log(t+1))^2}=\frac{\log(t+1)-t}{((t+1)\log(t+1))^2}.
\end{equation*} Thus, $\varphi(t) = tP''(t)/P'(t)$ is a decreasing function with the minimum value $\varphi(\infty)=p$ and the maximum value $\varphi(0)=p+1$.  On the other hand $g(t)=t+1/t$ is a strictly convex. Thus,
\begin{equation*}
c_P=\sup_{t\in(0,\infty)}g\left(\frac{tP''(t)}{P'(t)}\right)=\max\{g(p),g(p+1)\}=\max\left\{p+\frac{1}{p},p+1+\frac{1}{p+1}\right\}.
\end{equation*}

\subsection{Case of imaginary $z$}

Assume $z=ir$ with $r >0$. Then (\ref{localpp}) takes the form 
\begin{align}\label{locex1}
    	\left|\frac{tP''(t)}{P'(t)}+r^{2}\frac{tQ''(t)}{Q'(t)}-r^{2}-1\right|\leq \frac{tP''(t)}{P'(t)}-|z|^{2}\frac{tQ''(t)}{Q'(t)}+1-|z|^{2}
\end{align}
 holds for all $t>0$. After openning the absolute value with two different signs the  condition (\ref{locex1}) simplifies to 
 \begin{align}\label{mart1}
     r^{2}\leq  \frac{tP''(t)}{P'(t)} \quad \text{and} \quad \frac{tQ''(t)}{Q'(t)}r^{2}\leq 1 \quad \text{holds for all} \quad t>0.  
 \end{align}
 It follows from (\ref{mart1}) that $P''>0$. Since
 \begin{align*}
 F''(t) =  \frac{Q''(P^{-1}(t))}{(P'(P^{-1}(t)))^{2}} - \frac{Q'(P^{-1}(t))P''(P^{-1}(t))}{(P'(P^{-1}(t)))^{3}}   
 \end{align*}
 must be positive it follows that $Q''>0$. Thus (\ref{mart1}) implies 
 \begin{align*}
 r \leq \min \left\{\sqrt{\inf_{t>0}\frac{tP''(t)}{P'(t)}}, \sqrt{\inf_{t>0}\frac{Q'(t)}{tQ''(t)}} \right\}.
 \end{align*}
 So we obtain 
 \begin{corollary}\label{mainc}
     Let nonnegative $P,Q \in C^{2}((0, \infty))\cap C([0, \infty))$ be such that $P', P'', Q', Q''>0$, and both $P$ and $F = Q\circ P^{-1}$ satisfy the growth condition (\ref{growthC}). If $F$ satisfies (\ref{convexityp}) then we have 
     \begin{align}\label{globpav}
Q^{-1}\left(\int_{\mathbb{R}^{k}} Q(T_{ir^{*}}f) d\gamma\right) \leq P^{-1}\left( \int_{\mathbb{R}^{k}} P(f)d\gamma\right)
     \end{align}
     holds for all polynomials $f : \mathbb{R}^{k} \mapsto \mathbb{C}$, and all $k\geq 1$, where 
$$
r^{*} = \min\left\{1, \sqrt{\inf_{t>0}\frac{tP''(t)}{P'(t)}}, \sqrt{\inf_{t>0}\frac{Q'(t)}{tQ''(t)}} \right\}
$$
is the largest number in $[0, 1]$ for which (\ref{globpav}) hods true. 
 \end{corollary}
 
In the next section we present one simple way to generate functions $P$ and $Q$ which satisfy the assumptions in Theorem~\ref{mainth1} and Corollary~\ref{mainc}.

 \subsection{Systematic way to generate new examples}
 Pick any not identically zero, nonnegative,  concave $h \in C^{1}([0, \infty))$. Then it follows that $h(\theta)>0$ on $(0, \infty)$, and hence 
$ -\infty < \int_{1}^{s} \frac{d\theta}{h(\theta)}<\infty$ for all  $s>0$. 
Set 
\begin{align}\label{vax1}
    F(t) = \int_{0}^{t} \exp\left(\int_{1}^{s} \frac{d\theta}{h(\theta)} \right)ds. 
\end{align}
Notice that $F$ satisfies (\ref{convexityp}). Indeed, we have 
\begin{align*}
    F''(t) &= \frac{1}{h(t)}\exp\left(\int_{1}^{t} \frac{d\theta}{h(\theta)} \right)>0,\\
    \frac{F'(t)}{F''(t)} &=h(t) \quad \text{is concave on} \quad (0, \infty). 
\end{align*}
Next, take any $\varphi \in C^{1}(\mathbb{R})$ with $\inf_{t \in \mathbb{R}} \varphi'(t)>0$, and set 
\begin{align}
    P(t) &:=\int_{-\infty}^{\log(t)} e^{\varphi(s)+s}ds, \label{vax2}\\ 
    Q(t) &:= F(P(t)). \nonumber 
\end{align}
Notice that 
\begin{align*}
&P'(t), P''(t)>0, \quad \frac{tP''(t)}{P'(t)} = \varphi'(\log(t))>0, \quad  Q'(t) = F'(P(t))P'(t)>0, \\
&Q''(t) = F''(P(t))[P'(t)]^{2}+F'(P(t))P''(t)>0,\\
&\frac{tQ''(t)}{Q'(t)} =\frac{tP'(t)}{h(P(t))}+\frac{tP''(t)}{P'(t)} = \frac{te^{\varphi(\log(t))}}{h(P(t))}+\varphi'(\log(t))>0
\end{align*}
on $(0,\infty)$. Thus we have 
\begin{corollary}
Let $F$ and $P$ be defined as in (\ref{vax1}) and (\ref{vax2}), where $h\in C^{1}([0, \infty))$ is any nonneagtive concave $h \not\equiv 0$,  and $\varphi \in C^{1}(\mathbb{R})$ is any function with $\inf_{t \in \mathbb{R}} \varphi'(t)>0$. If both $F$ and $P$ satisfy the growth condition (\ref{growthC}) then 
$$
F\left(\int_{\mathbb{R}^{k}} P(|T_{z}f| )d\gamma\right)\leq \int_{\mathbb{R}^{k}} F(P(|f|))d\gamma
$$ 
holds for all polynomials $f :\mathbb{R}^{k} \mapsto \mathbb{C}$, and all $k\geq 1$ if and only if the local condition 
\begin{align}\label{marlocal}
\left|(1-z^{2})\left(\varphi'(s)-1\right)- z^{2}\frac{e^{\varphi(s)+s}}{h(P(e^{s}))}\right| \leq (1-|z|^{2})\left(\varphi'(s)+1\right)- |z|^{2}\frac{e^{\varphi(s)+s}}{h(P(e^{s}))}
\end{align}
holds for all $s \in \mathbb{R}$.
\end{corollary}

\textbf{Example 1.} If we let $\varphi(t)=(p-1)t$, and $h(t) = \frac{p}{q-p}t$, where $1<p<q<\infty$,  then clearly $F$ and $P$ satisfy the growth condition (\ref{growthC}), $h(P(e^{s}))=\frac{e^{sp}}{q-p}$, and we obtain that (\ref{marlocal}) reduces to (\ref{weis1}).

\textbf{Example 2}. Let again $\varphi(t)=(p-1)t, p>1$, and $h(t) = \alpha t+\log(1+t)$ with $\alpha>0$. Then clearly $F$ and $P$ satisfy the growth condition (\ref{growthC}). 
We have 
\begin{align*}
    \frac{tQ''(t)}{Q'(t)} =\frac{p}{\alpha + \frac{p\log(1+t^{p}/p)}{t^{p}}}+p-1.
\end{align*}
Thus $\sup_{t \in (0,\infty)} \frac{tQ''(t)}{Q'(t)}=\frac{p}{\alpha}+p-1$, and hence Corollary~\ref{mainc} holds with 
\begin{align*}
    r^{*} = \min\left\{\sqrt{p-1},  \frac{1}{\sqrt{\frac{p}{\alpha}+p-1}} \right\}.
\end{align*}

More sophisticated examples can be created by considering $\varphi(t) = \beta t - \log(1+t^{2}), \beta>1$, and $h(t) = \alpha t + t^{\lambda}$ with $\alpha>0$ and $\lambda \in (0,1)$.

\section{The construction of $C(s)$}\label{construction}

To explain the appearence of the map $C(s)$ for simplicity let $k=1$, and  let us rewrite the inequality (\ref{globalpp}) as follows 
\begin{align}\label{globalgaus}
\mathbb{E} F(P(|T_{z}f(\xi)|)) \leq F(\mathbb{E} P (|f(\xi)|)),
\end{align}
where $\xi \sim N(0,1)$ is the standard Gaussian random variable, and $F = Q\circ P^{-1}$. By the central limit theorem we may replace $\xi$ by 
\begin{align*}
\xi_{m} :=\frac{\varepsilon_{1}+\ldots+\varepsilon_{m}}{\sqrt{m}},
\end{align*}
where $\varepsilon_{1}, \ldots, \varepsilon_{m}$ are i.i.d. symmetric Bernoulli  $\pm 1$ random variables, and ask whether a ``stronger inequality'', i.e., (\ref{globalgaus}) holds with $\xi_{m}$ instead of $\xi$. Now notice that we can think of $f(\xi_{m})$ as a function on the Hamming cube $\{-1,1\}^{m}$, namely 
\begin{align*}
g(\varepsilon_{1}, \ldots, \varepsilon_{m}) = f\left(\frac{\varepsilon_{1}+\ldots+\varepsilon_{m}}{\sqrt{m}}\right).
\end{align*}
Any function $g : \{-1,1\}^{m} \mapsto \mathbb{C}$ admits Fourier--Wash series representaiton 
\begin{align}\label{walsh}
    g(\varepsilon) = \sum_{S \subset \{1, \ldots, m\}} \widehat{g}(S) W_{S}(\varepsilon), \quad \text{where} \quad W_{S}(\varepsilon) = \prod_{j \in S} \varepsilon_{j},
\end{align}
where $\varepsilon=(\varepsilon_{1}, \ldots, \varepsilon_{m}),$ and  $\widehat{g}(S) = \mathbb{E} g(\varepsilon) W_{S}(\varepsilon)$. 

An important part is to understand how $T_{z}f(\xi_{m})$ is related to the right hand side of (\ref{walsh}). Let $f(t) = \sum_{\ell=0}^{N} c_{\ell} H_{\ell}(t)$. Beckner's lemma \cite{Beckner} says that for any integer $\ell \geq 0$, we have
\begin{align*}
   H_{\ell}\left(\frac{\varepsilon_{1}+\ldots+\varepsilon_{m}}{\sqrt{m}}\right)  = \frac{\ell!}{m^{\ell/2}} \sum_{|S|=\ell, \, S \subset \{1, \ldots, n\}} W_{S}(\varepsilon) + L_{m}
\end{align*}
where $L_{m}$ are {\em lower order terms} which converge to zero in distribution. Thus 
\begin{align*}
g(\varepsilon_{1}, \ldots, \varepsilon_{n}) = f(\xi_{m}) = \sum_{S \subset \{1, \ldots, m\}, \, |S|\leq N}\frac{|S|!}{m^{|S|/2}}c_{|S|}W_{S}(\varepsilon) +\tilde{L}_{m},
\end{align*}
 and also
\begin{align*}
T_{z} f(\xi_{m}) = \sum_{\ell=0}^{N}z^{\ell}c_{\ell}H_{\ell}(\xi_{m}) = \sum_{S \subset \{1, \ldots, m\}, \, |S|\leq N}z^{|S|}\frac{|S|!}{m^{|S|/2}}c_{|S|}W_{S}(\varepsilon) +K_{m},
\end{align*}
where $K_{m}, \tilde{L}_{m}$ are some other lower order terms. If  we ignore  the lower order terms $\tilde{L}_{m}$ and $K_{m}$ and set them to be zero (at least we hope that this is the case in the limit $m\to \infty$ after taking the expectation), and we denote $a_{S} := \frac{|S|!}{m^{|S|/2}}c_{|S|}$, then we may wonder under what condiitons on $F, P$ and $z$  the following {\em discrete  version} of (\ref{globalgaus}) holds true 
\begin{align}\label{mainhyp}
\mathbb{E} F\big(P\big(|\sum_{S \subset\{1, \ldots, m\}} z^{|S|}a_{S} W_{S}(\varepsilon) |\big)\big) \leq F\big(\mathbb{E} P \big(|\sum_{S \subset\{1, \ldots, m\}} a_{S} W_{S}(\varepsilon)|\big)\big). 
\end{align}
If we denote $\tilde{g}(\varepsilon) =\sum_{S \subset\{1, \ldots, m\}} a_{S} W_{S}(\varepsilon)$, and extend  the  domain of the {\em polynomial} $\tilde{g}$ in a natural way to $\mathbb{C}^{m}$ we can rewrite (\ref{mainhyp}) as follows 
\begin{align}\label{hypnorm}
\mathbb{E} F\big(P\big(|\tilde{g}(z\varepsilon_{1}, \ldots, z\varepsilon_{m})|\big)\big) \leq F\big(\mathbb{E} P \big(|\tilde{g}(\varepsilon_{1}, \ldots, \varepsilon_{m})|\big)\big). 
\end{align}
The inequality (\ref{hypnorm}) for $m=1$ is what is called {\em two point-inequality}
\begin{align}\label{two-pp}
    \frac{F(P(|a+bz|))+F(P(|a-bz|))}{2} \leq F\left(\frac{P(|a+b|)+P(|a-b|)}{2}\right)
\end{align}
holds for all $a,b \in \mathbb{C}$. The verification of (\ref{two-pp}) for a given $F$, $P$, and $z\in \mathbb{C}$, can be a difficult techincal problem, see for example,  Weissler's conjecture \cite{IvNa}  where the two-point inequality (\ref{two-pp}) was solved in the case $F(t)=t$ and $P(t)=t^{p}$. 

An important observation is that if we let $a=1$, and $b=\omega \delta$, where $\omega \in \mathbb{C}$ is arbitrary and $\delta \in \mathbb{R}$, then after writing  Taylor's formula for both sides of the two-point inequality (\ref{two-pp}) at point $\delta=0$ we will see that the first order terms will cancel each other. The comparison of the second order terms in (\ref{two-pp}) and then optimization over all $w \in \mathbb{C}$  gives the local condition (\ref{localpp}) in Theorem~\ref{mainth1}. So there is a hope that in the limit $m\to \infty$ only the local condition  will be needed, which happens to be true in addition with (\ref{convexityp}). 

The condition (\ref{convexityp}) allows us to induct (\ref{two-pp}) and obtain (\ref{mainhyp}). Indeed, as explained in Section~\ref{Introduction}, the condition (\ref{convexityp}) implies (in fact is equiavlent to under some mild assumptions on $F$) that
\begin{align}\label{polya1}
	\mathfrak{M}_{F}(h) = F^{-1}\left(\int_{\Omega} F (h) d\mu\right) \quad \text{is convex}
\end{align} 
as a functional acting on nonnegative functions $h : \Omega \mapsto [0, \infty)$, where $d\mu$ is a probability measure on the probability space $(\Omega, d\mu)$ equipped with Borel $\sigma$-algebra. Therefore, we have 
\begin{align}
    &F^{-1}(\mathbb{E} F\big(P\big(|\tilde{g}(z\varepsilon_{1}, \ldots, z\varepsilon_{m})|\big)\big)) =F^{-1}(\mathbb{E}_{m-1}\mathbb{E}^{1}F\big(P\big(|\tilde{g}(z\varepsilon_{1}, \ldots, z\varepsilon_{m})|\big)\big)) \nonumber\\
    &\stackrel{(\ref{two-pp})}{\leq}F^{-1}(\mathbb{E}_{m-1}F\big(\mathbb{E}^{1}P\big(|\tilde{g}(\varepsilon_{1}, \ldots, z\varepsilon_{m})|\big)\big)) \stackrel{(\ref{polya1})}{\leq}\mathbb{E}^{1} F^{-1}(\mathbb{E}_{m-1}F\big(P\big(|\tilde{g}(\varepsilon_{1}, \ldots, z\varepsilon_{m})|\big)\big))\nonumber \\
    &\stackrel{\mathrm{induction}}{\leq}\mathbb{E} P (|\tilde{g}(\varepsilon)|). \label{izrdeba}
\end{align}    
Here $\mathbb{E}_{n-1}$ and $\mathbb{E}^{1}$ take the average in the last $n-1$ and the first variables correspondingly. 

The inductive argument (\ref{izrdeba}) shows that if the two-point inequality (\ref{two-pp}) holds, and also the functional $\mathfrak{M}_{F}(\cdot)$ is convex then the discrete map 
\begin{align}\label{dismap}
 &\varphi(k) = \mathbb{E}^{k}F^{-1}(\mathbb{E}_{n-k}F(P(|\tilde{g}(\varepsilon_{1}, \ldots, \varepsilon_{k}, z\varepsilon_{k+1}, \ldots, z\varepsilon_{m})|))) \\
 &=\mathbb{E}^{k}F^{-1}(\mathbb{E}_{n-k}Q(|\tilde{g}(\varepsilon_{1}, \ldots, \varepsilon_{k}, z\varepsilon_{k+1}, \ldots, z\varepsilon_{m})|)) 
\end{align}
is nondecreasing for $0\leq k \leq m $. If we choose $k=k(m)$ so that $\lim_{m\to \infty} k/m=s \in [0,1]$ then we expect the discrete map (\ref{dismap}) to converge to 
\begin{align*}
r(s)  =\int_{\R}F^{-1}\left(\int_{\R}Q\Big(\Big|\int_{\R}\int_{\R}R(\sqrt{s}(u+iv)+z\sqrt{1-s}(x+iy))\,d\gamma(v)\,d\gamma(y)\Big|\Big)\,d\gamma(x)\right)\,d\gamma(u),
\end{align*}
where $R(t) = \sum_{\ell=0}^{N}a_{\ell}t^{\ell}$. The passage from the discrete map $\varphi(k)$ to continuous map $r(s)$ is a bit nontrivial, and it was proved in \cite{IvanVol} for functions $Q(t)=t^{q}$ and $F(t)=t^{q/p}$. Monotonicity of $\varphi(k)$ implies that $r(s)$ is nondecreasing. Since $r(0)=F^{-1}\left(\int_{\mathbb{R}} Q(|T_{z}f|)d\gamma\right)$ and $r(1)=\int_{\mathbb{R}}F^{-1}(Q(|f|)) d\gamma = \int_{\mathbb{R}} P(|f|) d\gamma$ (see Section \ref{Proof}) we recover (\ref{globalpp}). 

One can ingore all the heuristics used in this reasoning and try to directly show $r'(s)\geq 0$. The map $C(s)$ introduced in Section~\ref{Proof} is similar to $r(s)$  but different,
\begin{align*}
C(s)=\int_{\R}F\left(\int_{\R}P\Big(\Big|\int_{\R}\int_{\R} R(\sqrt{s}(u+iv)+z\sqrt{1-s}(x+iy))\,d\gamma(v)\,d\gamma(y)\Big|\Big)\,d\gamma(u)\right)\,d\gamma(x),
\end{align*}
for example, our flow $C(s)$ interpolates (\ref{globalpp}) in a 
 slightly different way, namely, we have  $C(0)=\int_{\mathbb{R}}F(P(|T_{z}f|))d\gamma$ and $C(1)=F\left(\int_{\mathbb{R}} P(|f|) d\gamma\right)$. It remains to extend the argument from $\mathbb{R}$ to $\mathbb{R}^{k}$ and hope that $C'(s)\geq 0$. Surprisingly, this turns out to be the case  with the smoothened version of $C(s)$, under the assumptions (\ref{convexityp}) and (\ref{localpp}).

\section{The proof of Theorem~\ref{mainth1}}\label{Proof}

Let  $B \in C^{2}((0, \infty))$. We say $B$ satisfies  growth condition if there exists $N, M>0$ such that 
\begin{align}\label{growthC}
|B''(t)|<\exp( \log^{N}(t+N)) 
\end{align}
holds for all $t\geq M$. Clearly by integrating the inequality (\ref{growthC}) we obtain similar bounds  (perhaps with different constants $N$ and $M$) for $|B(t)|$ and $|B'(t)|$ too.

\begin{lemma}\label{pp1}
Let nonnegative $F, M  \in C([0,\infty)) \cap C^{2}((0, \infty))$ satisfy growth condition (\ref{growthC}), and be such that $F', M'>0$ on $(0, \infty)$.  Assume 
\begin{align*}
    F''(x)>0 \quad \text{and} \quad x \mapsto F'(x)/F''(x) \quad \text{is concave on} \quad (0,\infty).
\end{align*}
If the local condition 
\begin{align}\label{localp1}
-\frac{F''(M(x))}{F'(M(x))}(M'(x)\Re (zw))^{2} + M''(x)((\Re w)^{2} - (\Re zw)^{2}) + \frac{1-|z|^{2}}{2} \frac{M'(x)}{x}|w|^{2} \geq 0
\end{align}
holds for all $x>0$,  all $w \in \mathbb{C}$, and some $z \in \mathbb{C}, |z|\leq 1$,  then
\begin{align}\label{globalp1}
\int_{\mathbb{R}^{k}} F(M(|T_{z} f|^{2})) d\gamma \leq F \left(\int_{\mathbb{R}^{k}} M(|f|^{2}) d\gamma\right)
\end{align}
holds for all $k\geq 1$ and all polynomials $f:\mathbb{R}^{k} \mapsto \mathbb{C}$. 
\end{lemma}
\begin{proof}
Pick any $\varepsilon >0$ and consider $\tilde{M}(x) = M(x+\varepsilon)$. The main reason we introduce $\tilde{M}$ is to avoid issues with the values $M'(0)$, $M''(0)$ being not defined, and the fact that $M(0)$ can be zero which can make $F'(M(0))$ and $F''(M(0))$ undefined. However, for the first time of reading the proof the reader should ignore these techincal difficulties and think $\varepsilon =0$, i.e., $\tilde{M} = M$. 

Notice that $\tilde{M}>M(\varepsilon)>0$, and $\tilde{M}$ satisfies growth condition (\ref{growthC}). Also notice that $\tilde{M}$ satisfies the local condition (\ref{localp1}). Indeed, it suffices to apply the inequality (\ref{localp1}) at point $x+\varepsilon$ instead of $x$, and use the inequality 
$$
\frac{1-|z|^{2}}{2} \frac{M'(x+\varepsilon)}{x+\varepsilon} |w|^{2} \leq \frac{1-|z|^{2}}{2} \frac{\tilde{M}'(x)}{x} |w|^{2}. 
$$

Let $d\gamma^{(s)}$ denote the  Gaussian measure on $\R^k$ of variance $s>0$ by
\begin{equation*}
	d\gamma^{(s)}(x)=\frac{1}{(2\pi s)^{\frac{k}{2}}}e^{-|x|^2/2s}\,dx. 
\end{equation*}
 Let $f(x) = \sum_{|\alpha|\leq \mathrm{deg}(f)}c_{\alpha}H_{\alpha}(x)$, and let 
\begin{align*}
		\ell(x)=\sum_{|\alpha|\leq \deg f}c_{\alpha}x_1^{\alpha_1}\ldots x_k^{\alpha_k}
	\end{align*}
Next, define 
\begin{equation*}	g(x,u,s)=\int_{\R^k}\int_{\R^k}\ell((u+iv)+z(x+iy))\,d\gamma^{(s)}(v)\,d\gamma^{(1-s)}(y),
\end{equation*}
and consider $C(s):[0,1]\to \R$ given by 
	\begin{equation*} C(s)=\int_{\R^k}F\left(\int_{\R^k}\tilde{M}(|g(x,u,s)|^2)\,d\gamma^{(s)}(u)\right)\,d\gamma^{(1-s)}(x).
	\end{equation*}
	Note that the inequality (\ref{globalp1}) can be rewritten as $C(0)\leq C(1)$. To see this, let us make a change of variables and write 
	\begin{equation*}
		g(x,u,s)=\int_{\R^k}\int_{\R^k}\ell((u+i\sqrt{s}v)+z(x+i\sqrt{1-s}y))\,d\gamma(v)\,d\gamma(y).
	\end{equation*}
	Substitute $g(x,u,s)$ into the definition of $C(s)$ and write 
	\begin{align*}	C(s)=\int_{\R^k}F\left(\int_{\R^k}\tilde{M}\Big(\Big|\int_{\R^k}\int_{\R^k}\ell(\sqrt{s}(u+iv)+z\sqrt{1-s}(x+iy))\,d\gamma(v)\,d\gamma(y)\Big|^2\Big)\,d\gamma(u)\right)\,d\gamma(x).
	\end{align*}
	Notice that
	\begin{align*}
T_zf(x)&=\int_{\R^k}\ell(z(x+iy))\,d\gamma(y),\\
		f(u)&=\int_{\R^k}\ell(u+iv)\,d\gamma(v),
	\end{align*}
	and so,
	\begin{align*}
	C(0)&=\int_{\R^k}F\left(\tilde{M}\Big(\Big|\int_{\R^k}\ell(z(x+iy))\,d\gamma(y)\Big|^2\Big)\right)\,d\gamma(x)=\int_{\R^k}F(\tilde{M}(|T_zf|^2))\,d\gamma,
	\end{align*}
	and
	\begin{align*} C(1)=F\left(\int_{\R^k}\tilde{M}\Big(\Big|\int_{\R^k}\ell(u+iv)\,d\gamma(v)\Big|^2\Big)\,d\gamma(u)\right)=F\left(\int_{\R^k}\tilde{M}(|f|^2)\,d\gamma\right).
	\end{align*}
We claim $C'(s) \geq 0$. Denote by $I_sB$ the heat flow of  $B:\R^k \mapsto \mathbb{C}$ evaluated at zero, namely 
	\begin{equation}\label{heat flow}
		I_sB=\int_{\R^k} 
		B\,d\gamma^{(s)}.
	\end{equation}
Note that by the heat equation we have
	\begin{equation}\label{heat'-identity1}
		\frac{d}{ds}I_sB=\frac{1}{2}I_s(\Delta_x B).
	\end{equation}
	For what follows, we need to indicate the variable in which the flow $I_s$ is applied. For instance, in \eqref{heat'-identity1}, we will denote it as $I_s^x$ instead of $I_s$. With this notation we have
	\begin{align*}
		C'(s)&=\frac{d}{d s}\left(I_{1-s}^xF(I_s^u\tilde{M}(|g(x,u,s)|^2))\right)\\
		&=-\frac{1}{2}I_{1-s}^x\Delta_x [F(I_s^u\tilde{M}(|g(x,u,s)|^2))]+I_{1-s}^x\frac{d}{d s}F(I_s^u\tilde{M}(|g(x,u,s)|^2)).
	\end{align*}
	Denote 
 $$
 A=I_s^u\tilde{M}(|g(x,u,s)|^2),
 $$
 and let us shortly write $g:=g(x,u,s)$. We have 
		\begin{align*}
		C'(s)=I_{1-s}^x\left[-\frac{1}{2}\Delta_x [F(A))]  +F'(A))  I_s^u\left[\frac{1}{2}\Delta_u \tilde{M}(|g|^{2})+\frac{d}{ds}\tilde{M}(|g|^{2}) \right]\right].
	\end{align*}
	Set $g_j$ to be the partial derivative of $g$ with respect to $u_j$. This derivative indicates that we first differentiate the polynomial $\ell$ in the $j$-th coordinate and then apply the flow $I_{1-s}^yI_s^v$. Our goal is to express the quantity $C'(s)$ in terms of $g_{j}$ {\em as much as it is possible}, because later we can treat the variables $g_{j}$ as being arbitrary and it will help us to better see why $C'(s)$ has a definite sign.  
 
We have
	\begin{align*}
		g(x,u,s)_{x_j}&=zg_j,\\
		(|g|^{2})_{u_j}&=(g\overline{g})_{u_j}=g_i\bar{g}+g\bar{g_j}=2\Re(g_j\bar{g}),\\
		(|g|^{2})_{x_j}&=(g\overline{g})_{x_j}=g_{x_j}\bar{g}+g\bar{g}_{x_j}=zg_j\bar{g}+g\bar{z}\bar{g}_j=2\Re(zg_j\bar{g}),\\  (|g|^{2})_{u_ju_j}&=g_{jj}\overline{g}+g\overline{g_{jj}}+2|g_j|^2,\\		(|g|^{2})_{x_jx_j}&=g_{x_jx_j}\bar{g}+g\overline{g_{x_jx_j}}+2g_{x_j}\bar{g}_{x_j}=z^2g_{jj}\overline{g}+g\overline{z}^2\overline{g_{jj}}+2|z|^2|g_j|^2.
	\end{align*}
	We also compute,
	\begin{align*}
		\frac{d}{ds}g&=\frac{d}{ds}I_{1-s}^yI_s^v\ell((u+iv)+z(x+iy)) \\
		&=-\frac{1}{2}I_{1-s}^yI_s^v\Delta_y\ell((u+iv)+z(x+iy))+\frac{1}{2}I_{1-s}^yI_s^v\Delta_v\ell((u+iv)+z(x+iy))\\
		&=\frac{1}{2}\sum_{j=1}^k\left(-(iz)^2g_{jj}+i^2g_{jj}\right)=\frac{z^2-1}{2}\sum_{j=1}^kg_{jj},
	\end{align*}
	and therefore we have
	\begin{align*}
	\frac{d}{ds}\tilde{M}(|g|^{2})&=\tilde{M}'(|g|^{2})(|g|^{2})_s=\tilde{M}'(|g|^{2})\left(\bar{g}\frac{d}{ds}g+g\frac{d}{ds}\bar{g}\right)\\
	&=\sum_{j=1}^k\tilde{M}'(|g|^{2})\left(\frac{z^2-1}{2}\bar{g}g_{jj}+\frac{\bar{z}^2-1}{2}g\bar{g}_{jj}\right).
	\end{align*}
Thus
	\begin{align*}
	&C'(s)=\\
 &\frac{1}{2}\sum_{j=1}^kI_{1-s}^x\left[-\frac{\partial^2}{\partial x_j^2} F(A)+F'(A)  I_s^u\Big[\frac{\partial^2}{\partial u_j^2} \tilde{M}(|g|^{2})+\tilde{M}'(|g|^{2})[(z^2-1)\bar{g}g_{jj}+(\bar{z}^2-1)g\bar{g}_{jj}]\Big] \right].
\end{align*}
Let us denote by   $\mathcal{L}:=\mathcal{L}(x,s)$ the function inside the square brackets above. We will show that $\mathcal{L}\geq0$.	We have
	\begin{align*}
	&\mathcal{L}=\\
 &-F''(A)(I_s^u\tilde{M}'(|g|^{2})(|g|^{2})_{x_j})^2-F'(A)I_s^u\tilde{M}''(|g|^{2})((|g|^{2})_{x_j})^2-F'(A)I_s^u\tilde{M}'(|g|^{2})(|g|^{2})_{x_jx_j}\\ 
	&+F'(A)  I_s^u\big[\tilde{M}''(|g|^{2})((|g|^{2})_{u_j})^2+\tilde{M}'(|g|^{2})[(|g|^{2})_{u_ju_j}+(z^2-1)\bar{g}g_{jj}+(\bar{z}^2-1)g\bar{g}_{jj}]\big] ,
\end{align*}
so,  
	\begin{align*}
		\mathcal{L}=&-F''(A)(I_s^u\tilde{M}'(|g|^{2})(|g|^{2})_{x_j})^2+F'(A)I_s^u\tilde{M}''(|g|^{2})\Big[((|g|^{2})_{u_j})^2-((|g|^{2})_{x_i})^2\Big]+\\ 
		&+F'(A)I_s^u\tilde{M}'(|g|^{2})\Big[(|g|^{2})_{u_ju_j}-(|g|^{2})_{x_jx_j}+(z^2-1)\bar{g}g_{jj}+(\bar{z}^2-1)g\bar{g}_{jj}\Big].
	\end{align*} 
	Note that
	\begin{align*}
		(|g|^{2})_{u_ju_j}-(|g|^{2})_{x_jx_j}=g_{jj}\overline{g}(1-z^2)+g\bar{g}_{jj}(1-\bar{z}^2)+2|g_j|^2(1-|z|^2),
	\end{align*}
	Miraculously the second derivatives $g_{jj}$ cancel each other, and we get 
	\begin{align}\label{closetocondition}
			\mathcal{L}&=-4 F''(A)(I_s^u\tilde{M}'(|g|^{2})\Re (z g_{j} \overline{g}))^2&\\
&+F'(A)I_s^u \left( 4\tilde{M}''(|g|^{2})\Big[(\Re ( g_{j} \overline{g}))^2-(\Re (z g_{j} \overline{g}))^2\Big]+ 2(1-|z|^2)\tilde{M}'(|g|^{2})|g_j|^2 \right).\nonumber
	\end{align}
 Since $F'(A)>0$ it is enough to show $\frac{\mathcal{L}}{F'(A)}\geq 0$. Recall that $A = I_{s}^{u}\tilde{M}(|g|^{2})$,  $x \mapsto \frac{F'(x)}{F''(x)}$ is concave on $(0,\infty)$, and the map $(x,y) \mapsto \frac{y^{2}}{x}$ is convex on $\mathbb{R}\times (0, \infty)$. Therefore we have 
 \begin{align*}
 &-4 \frac{F''(A)}{F'(A)}(I_s^u\tilde{M}'(|g|^{2})\Re (z g_{j} \overline{g}))^2 = -4 \frac{(I_s^u\tilde{M}'(|g|^{2})\Re (z g_{j} \overline{g}))^2}{ \frac{F'(I_{s}^{u}\tilde{M}(|g|^{2}))}{F''(I_{s}^{u}\tilde{M}(|g|^{2}))}} \\
 &\stackrel{\mathrm{Jensen}}{\geq}-4 \frac{(I_s^u\tilde{M}'(|g|^{2})\Re (z g_{j} \overline{g}))^2}{ I_{s}^{u} \left( \frac{F'(\tilde{M}(|g|^{2}))}{F''(\tilde{M}(|g|^{2}))}\right)}\stackrel{\mathrm{Jensen}}{\geq}-4 I_{s}^{u}\left( \frac{(\tilde{M}'(|g|^{2})\Re (z g_{j} \overline{g}))^2}{  \frac{F'(\tilde{M}(|g|^{2}))}{F''(\tilde{M}(|g|^{2}))}}\right).
 \end{align*}
 \begin{remark}\label{remarka1}
 Notice that the inequality
 \begin{align*}
   -4 \frac{F''(A)}{F'(A)}(I_s^u\tilde{M}'(|g|^{2})\Re (z g_{j} \overline{g}))^2  \geq -4 I_{s}^{u}\left( \frac{F''(\tilde{M}(|g|^{2}))(\tilde{M}'(|g|^{2})\Re (z g_{j} \overline{g}))^2}{  F'(\tilde{M}(|g|^{2}))}\right)  
 \end{align*}
 also follows from Jensen's inequality provided that $(t,y) \mapsto \frac{F''(t)}{F'(t)}y^{2}$ is convex on $(0, \infty)\times \mathbb{R}$. 
 \end{remark}
 Thus we have $\frac{\mathcal{L}}{F'(A)} \geq 4I_{s}^{u} \Psi$, where 
 \begin{align*}
\Psi &=  - (\tilde{M}'(|g|^{2})\Re (z g_{j} \overline{g}))^2 \frac{F''(\tilde{M}(|g|^{2}))}{F'(\tilde{M}(|g|^{2}))}+\tilde{M}''(|g|^{2})\Big[(\Re ( g_{j} \overline{g}))^2-(\Re (z g_{j} \overline{g}))^2\Big] \\
 &+ \frac{1-|z|^2}{2}\tilde{M}'(|g|^{2})|g_j|^2.
 \end{align*}
If $|g|=0$ then $\Psi  = \frac{1-|z|^2}{2}\tilde{M}'(0)|g_j|^2 \geq 0$. Otherwise we denote $|g|^{2}=x > 0$ and $w = g_{j} \overline{g} \in \mathbb{C}$ then (\ref{localp1}) implies  $\Psi \geq 0$ finishing the proof of $C'(s) \geq 0$.

Recall that the condition $C'(s)\geq 0$ for $s \in  (0,1)$ implies $C(0)\leq C(1)$, i.e.,  
\begin{align*}
\int_{\mathbb{R}^{k}} F ( M (|T_{z} f|^{2}+\varepsilon)) d\gamma \leq F \left(\int_{\mathbb{R}^{k}} M(|f|^{2}+\varepsilon)d\gamma\right)
\end{align*}
holds for all $\varepsilon>0$. As $\varepsilon \to 0$ by continuity we have $M(|f|^{2}+\varepsilon) \to M(|f|^{2})$ and $F(M(|T_{z}f|^{2}+\varepsilon)) \to F(M(|T_{z}f|^{2}))$ everywhere, therefore, the fact that $F$ and $M$ satisfy growth condition (\ref{growthC}) together with Lebesgue dominated convergence theorem completes the proof of Lemma~\ref{pp1}. 
\end{proof}

Before we proceed to the next lemma, we would like to point out  the following 
\begin{proposition}\label{utver1}
If $F\in C^{4}((0, \infty))$ and  $F', F''>0$ on $(0, \infty)$,  then the map 
\begin{align}\label{pirveli1}
t \mapsto \frac{F'(t)}{F''(t)} \quad \text{is concave on} \quad (0,\infty)
\end{align}
if and only if the map
\begin{align}\label{meore2}
B(t,y)  \mapsto \frac{y^{2}F''(t)}{F'(t)} \quad \text{is convex on} \quad (0, 
\infty)\times \mathbb{R}.
\end{align}
\end{proposition}
\begin{proof}
Indeed, (\ref{meore2}) holds if and only $\mathrm{det} (\mathrm{Hess}(B)) \geq 0$. We have 
\begin{align*}
\mathrm{det} (\mathrm{Hess}(B)) &=2y^{2} \left(\frac{F''(t)}{F'(t)}\frac{d^{2}}{dt^{2}} \frac{F''(t)}{F'(t)} - 2 \left(\frac{d}{dt} \frac{F''(t)}{F'(t)} \right)^{2} \right)\\
&\frac{2y^{2}}{(F^{(1)})^{3}} \left(F^{(1)}F^{(2)}F^{(4)}+
(F^{(2)})^{2}F^{(3)}-2F^{(1)}(F^{(3)})^{2} \right),
\end{align*}
where $F^{(j)}$ denotes $j$'th derivative of $F(t)$ with respect to $t$. Also notice that 
\begin{align*}
    \frac{d^{2}}{dt^{2}} \frac{F'(t)}{F''(t)} = -\frac{1}{(F^{(2)})^{3}}(F^{(1)}F^{(2)}F^{(4)}+(F^{(2)})^{2}F^{(3)}-2F^{(1)}(F^{(3)})^{2}).
\end{align*}
Thus  the inequality $\mathrm{det} (\mathrm{Hess}(B)) \geq 0$ is the same as  $\frac{d^{2}}{dt^{2}} \frac{F'(t)}{F''(t)} \leq 0$ finishing the proof of the proposition. 
\end{proof}

\begin{lemma}
 Let nonnegative $Q, P \in C^{2}((0, \infty))\cap C([0, \infty))$ be such that $Q', P'>0$ on $(0, \infty)$, and both $P$ and $F:=Q\circ P^{-1}$ satisfy growth condition (\ref{growthC}). If either
 \begin{align*}
    F''(x)>0 \quad \text{and} \quad x \mapsto F'(x)/F''(x) \quad \text{is concave on} \quad (0,\infty), 
\end{align*}
or $F'' \equiv 0$ on $(0,\infty)$, and the local condition 
\begin{align}\label{localp2}
		|wz|^{2} - (\Re\, wz)^{2} + \frac{xQ''(x)}{Q'(x)} (\Re\, wz)^{2} \leq |w|^{2} - (\Re\, w)^{2} + \frac{xP''(x)}{P'(x)} (\Re\, w)^{2}
\end{align}
 holds for all $x>0$, all $w\in \mathbb{C}$, and some $z \in \mathbb{C}$, $|z|\leq 1$,   then we have
 \begin{align}\label{globp2}
		Q^{-1}\left(\int_{\mathbb{R}^{n}} Q(|T_{z} f|) d\gamma \right) \leq P^{-1}\left(\int_{\mathbb{R}^{n}} P(|f|)d\gamma \right).
 \end{align}
 holds for all polynomials $f : \mathbb{R}^{k} \mapsto \mathbb{C}$, and all $k\geq 1$. 
\end{lemma}
\begin{proof}
    Let us refwrite (\ref{globp2}) as follows 
    \begin{align}\label{kideve}
\int_{\mathbb{R}^{k}} F ( M (|T_{z} f|^{2})) d\gamma \leq F \left(\int_{\mathbb{R}^{k}} M(|f|^{2})d\gamma\right),
    \end{align}
    where $M(x)=P(\sqrt{x})$, and $F(x) = Q\circ P^{-1}(x)$. By Lemma~\ref{pp1} the inequality (\ref{kideve}) holds if $M$ and $F$ satisfy the conditions in Lemma~\ref{pp1}. All conditions trivially hold except one needs to check the local condition (\ref{localp1}). We claim that (\ref{localp1}) is the same as (\ref{localp2}). Indeed, we have 
    \begin{align*}
    &M'(x) = \frac{1}{2} \frac{P'(\sqrt{x})}{x^{1/2}},\\
    &M''(x) = P''(\sqrt{x}) \frac{1}{4x}-P'(\sqrt{x}) \frac{1}{4x^{3/2}},\\
    &F'(x) = \frac{Q'(P^{-1}(x))}{P'(P^{-1}(x))},\\
    &F''(x) = \frac{Q''(P^{-1}(x))}{(P'(P^{-1}(x)))^{2}}-\frac{Q'(P^{-1}(x))}{(P'(P^{-1}(x)))^{3}} P''(P^{-1}(x)).
    \end{align*}
Thus 
\begin{align*}
&\frac{F''(M(x))}{F'(M(x))} (M'(x))^{2} = 
 \left(\frac{Q''(\sqrt{x})}{P'(\sqrt{x})Q'(\sqrt{x} )}-\frac{P''(\sqrt{x})}{(P'(\sqrt{x}))^{2}}\right) (M'(x))^{2}\\
 &=\frac{1}{4x} \left(\frac{Q''(\sqrt{x})P'(\sqrt{x})}{Q'(\sqrt{x})}-P''(\sqrt{x})\right),
\end{align*}
and the local condition (\ref{localp1}) rewrites in terms of $P$ and $Q$ as follows
\begin{align}
&-\frac{1}{4x} \left(\frac{Q''(\sqrt{x})P'(\sqrt{x})}{Q'(\sqrt{x})}-P''(\sqrt{x})\right)(\Re zw)^{2}  \nonumber\\
&+\left(P''(\sqrt{x}) \frac{1}{4x}-P'(\sqrt{x}) \frac{1}{4x^{3/2}}\right)((\Re w)^{2} - (\Re zw)^{2}) \nonumber \\
&+\frac{1-|z|^{2}}{4} \frac{P'(\sqrt{x})}{x^{3/2}}|w|^{2} \geq 0. \label{simp1}
\end{align}

If we multiply both sides of the inequality (\ref{simp1}) by a positive number $\frac{4x^{3/2}}{P'(\sqrt{x})}$, and replace $\sqrt{x}$ by $x$ we obtain the inequality (\ref{localp2}). 
\end{proof}

\begin{lemma}
The local condition (\ref{localp2}) holds  if and only if 
\begin{align*}
    	\left|\left(\frac{xQ''(x)}{Q'(x)}-1\right)z^2-\left(\frac{xP''(x)}{P'(x)}-1\right)\right|\leq \left(\frac{xP''(x)}{P'(x)}+1\right)-\left(\frac{xQ''(x)}{Q'(x)}+1\right)|z|^2
\end{align*}
holds for all $x>0$. 
\end{lemma}
\begin{proof}
By homogeneity we can assume  $|w|=1$ in (\ref{localp2}), hence the local condition (\ref{localp2}) rewrites as  
\begin{equation*}
	|z|^2-1\leq\left(\frac{xP''(x)}{P'(x)}-1\right)(\Re w)^2-\left(\frac{xQ''(x)}{Q'(x)}-1\right)(\Re wz)^2.
\end{equation*}
Set $K=\frac{xP''(x)}{P'(x)}-1$ and $L=\frac{xQ''(x)}{Q'(x)}-1$. So for every  $|w|=1$ we have
\begin{equation*}
	2|z|^2-2\leq K\frac{w^2+2+\bar{w}^2}{2}-L\frac{(wz)^2+2|z|^2+(\bar{wz})^2}{2}.
\end{equation*}
This is the same as
\begin{equation*}
	\Re\left((Lz^2-K)w^2\right)=\frac{1}{2}(Lz^2-K)w^2+\frac{1}{2}(L\bar{z}^2-A)\bar{w}^2\leq (K+2)-(L+2)|z|^2.
\end{equation*}
Maximizing the left hand side over all $w \in \mathbb{C}, |w|=1,$ we obtain
\begin{equation*}
	\left|\left(\frac{xQ''(x)}{Q'(x)}-1\right)z^2-\left(\frac{xP''(x)}{P'(x)}-1\right)\right|\leq \left(\frac{xP''(x)}{P'(x)}+1\right)-\left(\frac{xQ''(x)}{Q'(x)}+1\right)|z|^2
\end{equation*}
completing the proof of the lemma. 
\end{proof}

\section{Necessity of the local condition (\ref{localpp})}
Assume $P, Q \in C^{3}((0, \infty))$. Rewrite the inequality (\ref{globalpp}) as 
\begin{align}\label{Lglobal}    \int_{\mathbb{R}^{k}}F(M(|T_{z}f|^{2}))d\gamma \leq F\left( \int_{\mathbb{R}^{k}} M(|f|^{2})d\gamma\right)
\end{align}
where $F = Q\circ P^{-1}$, and $M(x)=P(\sqrt{x})$. It sufficies to show that (\ref{Lglobal}) for all polynomials $f$ implies 
\begin{align}\label{Llocal}
-\frac{F''(M(x))}{F'(M(x))}(M'(x)\Re (zw))^{2} + M''(x)((\Re w)^{2} - (\Re zw)^{2}) + \frac{1-|z|^{2}}{2} \frac{M'(x)}{x}|w|^{2} \geq 0
\end{align}
holds for all $w \in \mathbb{C}$ and all $x>0$. The condition (\ref{Llocal}) is the same as (\ref{localp1}) and we have explained in Section~\ref{Proof} that this condition rewrites as the local condition (\ref{localpp}) after expressing the left hand side of (\ref{Llocal}) in terms of $P$ and $Q$, and optimizing over all $w \in \mathbb{C}$. 

Let $J(x):=F(M(x)))$. Recall that if $h \in C^{3}((0, \infty))$ then for each $t_{0}>0$ there exists $\delta = \delta(t_{0})>0$ such that 
\begin{align*}
h(t) = h(t_{0})+h'(t_{0})(t-t_{0})+\frac{h''(t_{0})}{2}(t-t_{0})^{2}+\mathcal{O}((t-t_{0})^{3})   
\end{align*}
holds for all $t>0$ with $|t-t_{0}|<\delta(t_{0})$. 

By taking $f$ to depend only on $x_{1}$ variable we can assume $k=1$ in (\ref{Lglobal}). Notice that if $h$ satisfies the growth condition (\ref{growthC}) then for any $N, C>0$ we have $\int_{|\varepsilon x|>C} |h(x)| d\gamma(x) = \mathcal{O}(\varepsilon^{N})$ as $\varepsilon \to 0$.  Fix $a,b \in \mathbb{C}$, $a\neq 0$, and let $f(x) = a+b\varepsilon x$, where $\varepsilon >0$. Clearly $T_{z} f(x)= a+zb\varepsilon x$. Thus we obtain that there exists $\delta = \delta(a,b,z)>0$ such that 
\begin{align*}
&\int_{\mathbb{R}} J(|a+zb\varepsilon x|^{2})d\gamma(x) = \\
&\int_{|\varepsilon x|\leq \delta} J(|a|^{2})+J'(|a|^{2})2 \Re (\bar{a}bz) \varepsilon x + (J''(|a|^{2})2(\Re (\bar{a}bz))^{2}+J'(|a|^{2}) |bz|^{2})|\varepsilon x|^{2} d\gamma(x) + \mathcal{O}(\varepsilon^{3}) =\\
&=J(|a|^{2})+(J''(|a|^{2})2(\Re (\bar{a}bz))^{2}+J'(|a|^{2}) |bz|^{2})\varepsilon^{2}+\mathcal{O}(\varepsilon^{3})
\end{align*}
as $\varepsilon$ goes to zero. Similarly, there exists $\tilde{\delta} = \tilde{\delta}(a,b)$ such that 
\begin{align*}
&\int_{\mathbb{R}} M(|a+b\varepsilon x|^{2})d\gamma(x) = \\
&\int_{|\varepsilon x|\leq \tilde{\delta}} M(|a|^{2})+M'(|a|^{2})2 \Re (\bar{a}b) \varepsilon x + (M''(|a|^{2})2(\Re (\bar{a}b))^{2}+M'(|a|^{2}) |b|^{2})|\varepsilon x|^{2} d\gamma(x) + \mathcal{O}(\varepsilon^{3}) =\\
&=M(|a|^{2})+(M''(|a|^{2})2(\Re (\bar{a}bz))^{2}+M'(|a|^{2}) |b|^{2})\varepsilon^{2}+\mathcal{O}(\varepsilon^{3}).    
\end{align*}
For $t_{0}=M(|a|^{2})$ we have 
\begin{align*}
 F(t_{0}+\varepsilon_{1}) = F(t_{0}) +F'(t_{0})\varepsilon_{1}+\mathcal{O}(\varepsilon_{1}^{2})
\end{align*}
as $\varepsilon_{1} \to 0^{+}$. Thus 
\begin{align*}
&F\left( \int_{\mathbb{R}} M(|a+b\varepsilon x|^{2}) d\gamma\right)   = \\
&F(M(|a|^{2}))+F'(M(|a|^{2}))(M''(|a|^{2})2(\Re (\bar{a}b))^{2}+M'(|a|^{2}) |b|^{2})\varepsilon^{2} + \mathcal{O}(\varepsilon^{3}).
\end{align*}
Therefore, the inequality (\ref{Lglobal}) implies 
\begin{align*}
&J(|a|^{2})+(J''(|a|^{2})2(\Re (\bar{a}bz))^{2}+J'(|a|^{2}) |bz|^{2})\varepsilon^{2}+\mathcal{O}(\varepsilon^{3}) \leq \\
&F(M(|a|^{2}))+F'(M(|a|^{2}))(M''(|a|^{2})2(\Re (\bar{a}b))^{2}+M'(|a|^{2}) |b|^{2})\varepsilon^{2} + \mathcal{O}(\varepsilon^{3}).
\end{align*}
Canceling the term $J(|a|^{2})$,  dividing both sides of the inequality by $\varepsilon^{2}$,  and taking $\varepsilon \to 0$ we obtain 
\begin{align}\label{vse1}
    J''(|a|^{2})2(\Re (\bar{a}bz))^{2}+J'(|a|^{2}) |bz|^{2} \leq F'(M(|a|^{2}))(M''(|a|^{2})2(\Re (\bar{a}b))^{2}+M'(|a|^{2}) |b|^{2}).
\end{align}
Denoting $w = \bar{a}b$, $|a|^{2}=x$, and multiplying the both sides of the inequality by $|a|^{2}$ we obtain 
\begin{align}\label{vse2}
    2J''(x)x (\Re (wz))^{2}+J'(x)|wz|^{2} \leq F'(M(x))(xM''(x)2(\Re(w))^{2}+M'(x))|w|^{2}.
\end{align}
After substituting the values $J'(x) = F'(M(x))M'(x)$ and $J''(x) = F''(M(x))(M'(x))^{2}+F'(M(x))M''(x)$  into (\ref{vse2}) we can rewrite (\ref{vse2}) as 
\begin{align}
    &2x[F''(M(x))(M'(x))^{2}+F'(M(x))M''(x)] (\Re (wz))^{2}+F'(M(x))M'(x)|wz|^{2} \leq \nonumber \\
    &F'(M(x))(xM''(x)2(\Re(w))^{2}+M'(x))|w|^{2}. \label{vse3}
\end{align}
Dividing both sides of the inequality (\ref{vse3}) by $2xM'(x)$ we obtain (\ref{Llocal}). Notice that $x=|a|^{2}>0$ and $w=\bar{a}b \in \mathbb{C}$ can be chosen to be arbitrary positive and complex number correspondingly, therefore, this completes the proof of the lemma.

\section*{Acknowledgements}
P.I. was supported in part by NSF CAREER-DMS-2152401.

\end{document}